\documentclass{amsart}

\usepackage{bbm}
\RequirePackage[OT1]{fontenc}
\RequirePackage{amsthm,amsmath}
\RequirePackage[numbers]{natbib}
\RequirePackage[colorlinks,citecolor=blue,urlcolor=blue]{hyperref}

\newtheorem{theorem}{Theorem}[section]
\newtheorem{prop}[theorem]{Proposition}
\newtheorem{lemma}[theorem]{Lemma}

\newtheorem{definition}[theorem]{Definition}

\newtheorem{defn}[theorem]{Definition}

\newtheorem{rmk}[theorem]{Remark}

\newtheorem{eg}[theorem]{Example}

\newtheorem{cond}[theorem]{Condition}

\newcommand\FF{{\mathcal F}}

\newcommand\LL{{L}}
\newcommand\MM{{\mathcal M}}

\newcommand\PP{{\mathcal P}}

\newcommand\PMF{{\PP\kern-2pt\MM\FF}}

\newcommand\PML{{\PP\kern-2pt\MM\LL}}

\newcommand\ep{\epsilon}

\newcommand{\sas}{{S\alpha S}}

\newcommand{\LLL}{L}

\newcommand{\mups}{{\mu^{PS}}}

\newcommand{\mubms}{{\mu^{BMS}}}

\newcommand{\om}{{\Omega}}

\newcommand{\ind}{\mathbf{1}}
\newcommand{\e}{\epsilon}

\newcommand\eqd{\stackrel{d}{=}}

\numberwithin{equation}{section}
\allowdisplaybreaks
\usepackage{amsfonts,amssymb,verbatim,amsmath,amsthm,latexsym,textcomp,amscd}
\usepackage{latexsym,amsfonts,amssymb,epsfig,verbatim}
\usepackage{amsmath,amsthm,amssymb,latexsym,graphics,textcomp, hyperref}
\usepackage{graphicx}
\usepackage{color}
\usepackage{url}
\usepackage{psfrag}

\begin{document}
	
	\title[Mixing Properties of Stable Fields]{Mixing Properties of Stable Random Fields Indexed by Amenable and Hyperbolic Groups}
	
	\author{Mahan Mj}
	\address{Mahan Mj, School of Mathematics, Tata Institute of Fundamental Research, 1 Homi Bhabha Road, Mumbai 400005, India}
	
	\email{mahan@math.tifr.res.in}
	\email{mahan.mj@gmail.com}
	
	\author{Parthanil Roy}
	\address{Parthanil Roy, Theoretical Statistics and Mathematics Unit, Indian Statistical Institute, 8th Mile, Mysore Road, RVCE Post, Bangalore 560059, India}
	
	\email{parthanil.roy@gmail.com}
	
	\author{Sourav Sarkar}
	\address{Sourav Sarkar, Centre for Mathematical Sciences, University of Cambridge, Wilberforce Road, Cambridge CB3~0WB}
	\email{ss2871@cam.ac.uk}
	
	\thanks{M.M.~is partially supported by a DST J~C~Bose Fellowship, an endowment from the Infosys Foundation, and by the Department of Atomic Energy, Government of India, under project no.12-R\&D-TFR-5.01-0500. M.M. also thanks  CNRS for support during a visit to Institut Henri Poincar\'e in the summer of 2022. P.R.~is partially supported by a DST SwarnaJayanti Fellowship, and a SERB grant MTR/2017/000513.}
	
	


\begin{abstract}
We show that any stationary symmetric $\alpha$-stable ($\sas$) random field indexed by a countable amenable group $G$ is weakly mixing if and only if it is generated by a null action, extending the work of Samorodnitsky and Wang-Roy-Stoev for abelian groups to all amenable groups. This enables us to improve significantly the domain of a recently discovered connection to von Neumann algebras. We also establish ergodicity of stationary $\sas$ fields associated with boundary and  double boundary actions of a hyperbolic group $G$, where the boundary is equipped with either the Patterson-Sullivan or the hitting measure of a random walk, and the double boundary is equipped with the Bowen-Margulis-Sullivan measure.
\end{abstract}

	\maketitle

\section{Introduction}\label{sec-intro}

This paper deals with various mixing properties of stationary symmetric stable random fields indexed by countably infinite groups, particularly amenable and hyperbolic ones. Random fields indexed by amenable groups (e.g., Heisenberg groups, discrete matrix groups, group of permutations of $\mathbb{N}$ with finite support, etc.) arise naturally in machine learning algorithms for structured and dependent space-time data \cite{austern:orbanz:2022}.  On the other hand, random fields indexed by hyperbolic groups are useful in tree-indexed processes, branching models, etc.~\cite{pemantle:1995,sarkar:roy:2018, athreya:mj:roy:2019}. Mixing properties of these fields are important tools for estimation of ruin probabilities for stable processes \cite{mikosch:samorodnitsky:2000a}, investigation of asymptotic properties of algorithms in the context of space-time statistical inference for max-stable random fields \cite{davis:kluppelberg:steinkohl:2013}, analysis of limiting behavior for functionals of L\'evy driven processes \cite{basse-oconnor:heinrich:podolskij:2019}, and so on.

Let $G$ be a countably infinite group and $\mathbf{Y}=\{Y_g\}_{g \in G}$ be a random field (i.e., a collection of random variables defined on a common probability space $(\Omega, \mathcal{A}, \mathbb{P})$) indexed by $G$. Such a random field is called \emph{symmetric $\alpha$-stable} ($\sas$) if each finite linear combination of $Y_g$'s follows an $\sas$ distribution.  A $\sas$ random field is called left-stationary if its law is invariant under the left-translation of its indices. Thanks to Rosi\'nski~\cite{rosinski:1995, rosinski:2000}, the joint distribution (and hence the dependence structure) of any left-stationary $\sas$ field is uniquely determined by a non-singular $G$-action, an $L^\alpha$ function and a $\pm 1$-valued cocycle (see Section~\ref{sec:sas} below). Among these, the non-singular action, being an  infinite-dimensional parameter of a left-stationary $\sas$ random field, contains a lot of information about various probabilistic properties of the field.

Keeping the above discussion in mind, it is not surprising that various  aspects of probability theory (e.g., growth of maxima \cite{samorodnitsky:2004a,roy:samorodnitsky:2008,owada:samorodnitsky:2015a,sarkar:roy:2018,athreya:mj:roy:2019}, extremal point processes \cite{resnick:samorodnitsky:2004,roy:2010a,sarkar:roy:2018},  mixing features  \cite{rosinski:samorodnitsky:1996,samorodnitsky:2005a,roy:2007a, roy:2012,wang:roy:stoev:2013}, path properties \cite{panigrahi:roy:xiao:2021}, large deviations \cite{mikosch:samorodnitsky:2000a,fasen:roy:2016}, functional central limit theorem  \cite{owada:samorodnitsky:2015b,jung:owada:samorodnitsky:2017},  etc.) of a stationary
$\sas$ random field have been linked to the ergodic theory of the underlying non-singular action. In this paper, we extend this connection  to the domain of geometric group theory, particularly amenable and hyperbolic groups. We also significantly extend the range of application of a recently discovered association with von Neumann algebras resolving completely
Problem~3 (and hence Problem~1 and Conjecture~2) of \cite{roy:2020}. We have also resolved Conjecture~5 (for hyperbolic groups, not just for free groups) of this reference positively in Section~\ref{sec-hyp}. 

We prove the two necessary and sufficient conditions for weak mixing of a left-stationary $\sas$ random field indexed by an amenable group - one is in terms of the ergodic theory of the underlying non-singular group action and the other is a von-Neumann algebraic characterization via the associated crossed product construction.
\begin{theorem}\label{t:mainintro}
	Suppose that $G$ is a countable  amenable group and $\mathbf{Y}=\{Y_g\}_{g \in G}$ is a stationary $\sas$ random field. Suppose that $\mathbf{Y}$ admits a Rosi\'{n}ski representation such that the underlying non-singular $G$-action $\{\phi_g\}$ is free. Then the following are equivalent.
	\begin{enumerate}
		\item  $\{Y_g\}_{g \in G}$ is generated by a null action (in some and hence all Rosi{\'n}ski representations).
		\item $\{Y_g\}_{g \in G}$ is weakly mixing;
		\item  the group measure space construction corresponding to $\{\phi_g\}_{g \in G}$ does not admit a $II_1$ factor in its central decomposition.
	\end{enumerate}
\end{theorem}
\noindent See Theorem \ref{t:main} for the equivalence of (1) and (2) (which does not require the assumption of freeness of the action), and Theorem \ref{thm:wm_charac_vNalg} for the equivalence of (2) and (3).

The main challenge in the proof of Theorem~\ref{t:mainintro} is ergodic theoretic - more precisely, the unavailability of an ergodic theorem for non-singular (but not necessarily measure preserving) actions of amenable groups even along a tempered F{\o}lner sequence.
We remove this obstacle with the help of a probabilistic truncation argument along with a theorem of Lindenstrauss \cite{lindenstrauss:2001} refined by Tempelman  \cite{tempelman:2015} and finally by applying the Maharam skew-product. 
This extends the main theorems of \cite{samorodnitsky:2005a, wang:roy:stoev:2013, roy:2020} for abelian groups to all amenable groups. A rather different proof (more analytic in nature) of the equivalence of (1) and (2) in Theorem \ref{t:mainintro} has  recently been given by Avraham-Re’em \cite{nachi}.

In Section \ref{sec-hyp}, we turn to discrete groups $G$ acting on (Gromov-)hyperbolic spaces $X$, giving rise to an action of $G$ on the boundary $\partial_G X$. We shall be specifically interested in the following cases (See Example \ref{eg-confden}):
\begin{enumerate}
	\item $\partial_G X$ is the Gromov boundary of a
	(Gromov-)hyperbolic group $G$ equipped with the Patterson-Sullivan measure $\mups$.
	\item More generally, we can let $\partial_G X$ be the limit set of a non-elementary group $G$ acting properly by isometries on a hyperbolic space $X$, and equip it with the Patterson-Sullivan measures $\mups$.
	\item A non-elementary hyperbolic group $G$ acting on its Poisson boundary\\ $(\partial_G G,\mu_p)$, where $\mu_p$ denotes the hitting measure of random walks on the Poisson boundary.
	\item Let $\rho$ be an Anosov representation of a 	(Gromov-)hyperbolic group $G$ in the sense of Labourie,    in a semi-simple Lie group $L$.  Equip the limit set $\Lambda$ in the Furstenberg boundary with a `higher rank' Patterson-Sullivan measure $\mups$.
\end{enumerate}

We prove (see Theorem \ref{thm-hyp}):

\begin{theorem}\label{thm-hypintro}
	In all the  above examples,  the associated $G$-indexed $\sas$ random field  is ergodic.
\end{theorem}
A similar ergodicity result  (see Proposition \ref{prop-db}) is proven for  a
(Gromov-)hyperbolic group $G$ acting on the double boundary
$\big((\partial_G X \times \partial_G X)\setminus \Delta\big)$ equipped with the Bowen-Margulis-Sullivan measure, where $\Delta$ denotes the diagonal.\\

\noindent {\bf Acknowledgment:}   We are grateful to Nachi Avraham-Reem for kindly bringing his  recent work \cite{nachi} to our notice as this paper was being written up. {We are also grateful to him for extremely helpful comments on  an earlier version of this paper.} In addition to a different proof of the equivalence of (1) and (2) in
Theorem \ref{t:mainintro}, he also establishes that these are equivalent to ergodicity of the associated $\sas$ field indexed by an amenable group. The second and third authors would like to acknowledge the warm hospitality of Tata Institute of Fundamental Research, Mumbai during their extremely productive research visit in April 2022. 

\subsection{Stable random fields indexed by countable groups}\label{sec:sas}

We shall summarize briefly some basic material on  stationary symmetric $\alpha$-stable random fields indexed by a countable group.  A (real-valued) random variable $Y$ is said to follow a {\emph symmetric $\alpha$-stable} ($\sas$) distribution with scale parameter $\sigma \in (0, \infty)$ and tail parameter $\alpha \in (0, 2]$ if its characteristic function (i.e., Fourier transform) is of the form $\mathbb{E}(e^{i\theta Y}) = \exp{\{-\sigma^{\alpha}|\theta|^\alpha\}}$, $\theta \in \mathbb{R}$. Here $\mathbb{E}$ denotes the expectation with respect to the probability measure $\mathbb{P}$. Clearly, the distribution of $Y$ is the same as that of $-Y$, i.e.\ it is  \emph{symmetric}. When  $\alpha =2$, we get a centered normal random variable with exponential tails. Here, our focus is on the non-Gaussian case ($0 < \alpha < 2$). Thus, $\mathbb{P}(Y>t) = \mathbb{P}(Y<-t) \sim ct^{-\alpha}$ as $t \to \infty$. 

Let $G$ be a countably infinite group with identity element $e$.
A random field $\mathbf{Y}=\{Y_g\}_{g\in G}$ indexed by $G$ is called an $\sas$ random field if for each $k \geq 1$,  for all
$g_1, g_2, \ldots, g_k \in G$ and for all  $c_1, c_2, \ldots, c_k \in \mathbb{R}$, the linear combination $\sum_{i=1}^k c_i Y_{g_i}$  follows an $\sas$ distribution, whose scale parameter depends on the coefficients $c_1, c_2, \ldots, c_k \in \mathbb{R}$ as follows (see \cite{samorodnitsky:taqqu:1994} for further details):

\begin{equation}
	\sum_{i=1}^k c_i Y_{g_i} \sim \sas\bigg(\Big\|\sum_{i=1}^k c_i f_{g_i}\Big\|_\alpha\bigg)\,. \label{distn_lin_comb}
\end{equation}
Equivalently, any $\sas$ random field $\mathbf{Y}=\{Y_g\}_{g\in G}$ admits an \emph{integral representation} (also known as a \emph{spectral representation}) of the form
\begin{equation}\label{integrep}
	Y_g \eqd \int_S f_g(x)M(dx), \mbox{ \ \  } g \in G,
\end{equation}
where $M$ is an S$\alpha$S random measure on some $\sigma$-finite standard Borel space $(S,\mathcal{S}, \mu)$, and $f_g \in L^\alpha (S, \mu)$, $g \in G$ are real valued functions~\cite[Theorem~13.1.2]{samorodnitsky:taqqu:1994}. 
Conversely, given any $\sigma$-finite standard Borel space $(S,\mathcal{S}, \mu)$ and a collection of real-valued functions $\{f_g: g \in G\} \subset L^\alpha (S, \mu)$, there exists a stationary $\sas$ random field $\mathbf{Y}=\{Y_g\}_{g \in G}$ satisfying \eqref{distn_lin_comb}. Assume henceforth, without loss of generality, that $\bigcup_{g \in G} \{x \in S: f_g(x) \neq 0\} = S$
holds modulo $\mu$-null sets.

\begin{defn}
	The random field $\{Y_g\}_{g \in G}$ is called {\it left-stationary}
	if $\{Y_g\} \overset{d}{=} \{Y_{hg}\} $ for all $h\in G$, i.e., for all $k \in \mathbb{N}$, for all $g_1, g_2, \ldots, g_k, h \in G$ and for all Borel $B \subseteq \mathbb{R}^k$, 
	$$\mathbb{P}\big[(Y_{g_1},Y_{g_2}, \cdots, Y_{g_k}) \in B \big] = \mathbb{P}\big[(Y_{hg_1},Y_{hg_2}, \cdots, Y_{hg_k}) \in B \big].$$
\end{defn}
\noindent We will simply write stationary to mean left-stationary throughout this paper. 
Note that stationarity of $\mathbf{Y}=\{Y_g\}_{g\in G}$ is equivalent to saying that the left translation action $G \curvearrowright \mathbb{R}^G$ preserves the following probability measure
$$\mathbb{P}_{\mathbf{Y}} = \mbox{\;law of\;} \mathbf{Y}\,:=\mathbb{P}\Big(\big\{\omega \in \Omega: \big(X_g(\omega): g \in G\big) \in \cdot\big\}\Big).$$
In particular, $\big(\mathbb{R}^G, \,  \mathcal{B}_\mathbb{R}^{\otimes G}, \, \mathbb{P}_{\mathbf{Y}}, \, G\big)$ is a probability measure-preserving dynamical system whenever $\mathbf{Y}=\{Y_g\}_{g\in G}$ is stationary. 

By the pioneering work of Rosi{\'n}ski \cite{rosinski:1995, rosinski:2000}, this 
induces a \emph{non-singular} $G$-action as follows. 
Let $(S, \mathcal{S}, \mu)$ be a $\sigma$-finite standard measure space and $G \curvearrowright (S, \mathcal{S}, \mu)$ be a measurable action, i.e., for all $g \in G$, the map $\phi_g: S \to S$ defined by
$
\phi_g: x \mapsto g^{-1}.\,x
$
is measurable. This $G$-action is said to be \emph{measure-preserving} (or \emph{invariant}) if for all $g \in G$ and for all $A \in \mathcal{S}$,
$
g_\ast \mu (A):= \mu \circ \phi_g^{-1}(A) = \mu(A).
$  The action $G \curvearrowright (S, \mathcal{S}, \mu)$
is called \emph{non-singular} (or \emph{quasi-invariant}  or \emph{measure-class preserving}) if for all $g \in G$ and for all $A \in \mathcal{S}$,
$
g_\ast \mu (A)=0 \;\; \mbox{ if and only if } \;\; \mu(A) =0
$ (see \cite{aaronson:1997}). 
Following \cite{rosinski:1995}, we shall refer to the collection  $\{\phi_g:S \to S\}_{g \in G}$ as a non-singular $G$-action throughout this paper. For such an action, a $\pm 1$-valued measurable cocycle $c_g: S \to \{+1, -1\}$, $g \in G$ is another of measurable maps (also indexed by $G$) satisfying the cocycle relation 
$
c_{g_1g_2}(x) = c_{g_2} (x)c_{g_1}(\phi_{g_2}(x))
$
for all $g_1, g_2 \in G$ and for $\mu$-almost all $x \in S$. 
The following key theorem forms the link between stationary $\sas$ random fields and non-singular actions.

\begin{theorem}[\label{omni-sas}\emph{Rosi{\'n}ski Representation}~\cite{rosinski:1995, rosinski:2000}]${}$\\
	(a) For any stationary $\sas$ random field $\mathbf{Y}=\{Y_g\}_{g \in G}$, there exist a standard measure space $(S, \mathcal{S}, \mu)$ equipped with a non-singular group action $\{\phi_g\}_{g \in G}$, a $\pm 1$-valued measurable cocycle $\{c_g\}_{g\in G}$ and a real-valued function $f=f_e \in \LL^\alpha (S,\mu)$ such that $\mathbf{Y}$ admits an integral representation  of the form 
	\begin{equation}\label{eqthm}
		f_g(x)=c_g(x)\left(\frac{d(\mu\circ\phi_g)}{d\mu}(x)\right)^{1/\alpha}\left( f\circ \phi_g\right)(x), \mbox{\ \ } g \in G.
	\end{equation}	
	(b) Conversely, given any standard measure space $(S, \mathcal{S}, \mu)$ endowed with a non-singular group action $\{\phi_g\}_{g \in G}$, a $\pm 1$-valued measurable cocycle $\{c_g\}_{g\in G}$ for $\{\phi_g\}$, and a real-valued function $f=f_e \in \LL^\alpha (S,\mu)$, there exists a stationary $\sas$ random field indexed by $G$ admitting a Rosi{\'n}ski representation \eqref{eqthm}.
\end{theorem}

\section{Conditions for weak mixing}\label{sec-wm}
As discussed in Section~\ref{sec:sas}, a stationary random field $\{Y_g\}_{g \in G}$ (not necessarily $\sas$) induces a measure-preserving dynamical system $\big(\mathbb{R}^G, \,  \mathcal{B}_\mathbb{R}^{\otimes G}, \, \mathbb{P}_{\mathbf{Y}}, \, G\big)$. We say that $\{Y_g\}_{g \in G}$ is ergodic if the induced dynamical system is so, i.e., the (left) translation invariant measurable subsets of $\mathbb{R}^G$ have $\mathbb{P}_{\mathbf{Y}}$ measure $0$ or $1$. Roughly speaking, ergodicity is a very weak form of asymptotic independence of the random field. A slightly stronger notion is that of weak mixing, which is defined below.

\begin{defn} A   stationary random field $\{Y_g\}_{g \in G}$ (not necessarily $\sas$) is called weakly mixing if the dynamical system induced by the diagonal action $G \curvearrowright (\mathbb{R}^G \times \mathbb{R}^G, \,  \mathcal{B}_\mathbb{R}^{\otimes G} \, \otimes \, \mathcal{B}_\mathbb{R}^{\otimes G} , \, \mathbb{P}_{\mathbf{Y}} \otimes \mathbb{P}_{\mathbf{Y}})$ defined by 
	\[
	g.(x_1, x_2) = (g.x_1, g.x_2),\;\;\;\;\;\;\;\;\;\;\;\;\;\; g \in G, \, (x_1, x_2) \in \mathbb{R}^G \times \mathbb{R}^G
	\]	
	is ergodic.
\end{defn}

In this paper, we introduce a new notion of mixing that depends on a sequence $\mathcal{F}=\{F_n: n \geq 1\}$ of increasing exhaustive finite subsets of $G$.

\begin{defn}\label{defn:F_mixing}
	Suppose $G$ is a countable group and $\mathcal{F}=\{F_n: n \geq 1\}$ is a sequence of finite subsets $F_n \uparrow G$. We say that a stationary random field $\{Y_g\}_{g \in G}$ (not necessarily $\sas$) is $\mathcal{F}$ mixing if for all $k \geq 1$, for all $g_1, g_2, \ldots, g_k, h \in G$ and for all Borel $A, B \subset \mathbb{R}^k$,
	\begin{align*}
		&\frac{1}{|F_n|}\sum_{h \in F_n}\Big|\mathbb{P}\big[ (Y_{hg_1}, Y_{hg_2}, \ldots, Y_{hg_k}) \in A , \; (Y_{g_1}, Y_{g_2}, \ldots, Y_{g_k}) \in B \big] \\
		&\;\;\;\;\;\;\;\;\;\;\;\;\;\;\;\;\;\;\;\;\;\;\;\; - \mathbb{P}\big[ (Y_{g_1}, Y_{g_2}, \ldots, Y_{g_k}) \in A \big] \, \mathbb{P}\big[ (Y_{g_1}, Y_{g_2}, \ldots, Y_{g_k}) \in B \big]\Big| \to 0
	\end{align*}
	as $n \to \infty$
\end{defn}

\noindent Clearly, $\mathcal{F}$-mixing (for some $\mathcal{F}$) implies ergodicity; further, it is equivalent to weak mixing when $G$ is amenable and $\mathcal{F}$ is a F{\o}lner sequence; see Theorem~\ref{thm:wm_amenable} below.

Recall that a countable group $G$ is amenable if and only if it admits an increasing F{\o}lner sequence $F_n \uparrow G$, i.e., an increasing sequence of exhausting finite subsets $F_n \subset G$ such that for all $g \in G$,
\[
\lim_{n \to \infty} \frac{|gF_n \, \Delta \, F_n|}{|F_n|} = 0.
\]
Examples of amenable groups include finite groups, abelian groups, groups of sub-exponential growth, solvable groups, lamplighter groups, etc. As mentioned in the introduction, random fields indexed by amenable groups are important in machine learning algorithms for structured and dependent data \cite{austern:orbanz:2022}. 

In order to state the characterization result for weak mixing of stationary $\sas$ random fields indexed by amenable groups, we need to introduce the Neveu decomposition. Let $(S, \mathcal{S}, \mu)$ be as in Theorem \ref{omni-sas}, equipped with a non-singular group action $\{\phi_g\}_{g \in G}$, where $G$ is any countable group. By applying  \cite[Lemma~2.2, Theorem 2.3~(i)]{wang:roy:stoev:2013} in this general setup, we get a partition (known as the Neveu decomposition) $S=\mathcal{P}\cup \mathcal{N}$, where the set $\mathcal{P}$ is the largest (modulo $\mu$) $G$-invariant set where one can have a finite $G$-invariant measure equivalent to $\mu$, and $\mathcal{N}=\mathcal{P}^c$.  The subsets $\mathcal{N}$ and $\mathcal{P}$ of $S$ are known as the \emph{null part} and the \emph{positive part} of $\{\phi_g\}_{g \in G}$, respectively. A measurable set $W \subset S$ is called \emph{weakly wandering} if there exists an infinite subset $L \subseteq G$ such that $\{\phi_g(W): g \in L\}$ is a pairwise disjoint collection. It can be shown that $\mathcal{N}$ is the measurable union (and hence a countable union) of weakly wandering subsets of $S$ while $\mathcal{P}$ has no weakly wandering subset of positive $\mu$-measure; see \cite{aaronson:1997}. It was shown in \cite{samorodnitsky:2005a} that for $G=\mathbb{Z}$, a stationary $\sas$ process is weakly mixing if and only if the underlying non-singular $\mathbb{Z}$-action has no nontrivial positive part. In \cite{wang:roy:stoev:2013}, this result was generalized to $G = \mathbb{Z}^d$ for any $d \in \mathbb{N}$ with the help of \cite{takahashi:1971}.  In Theorem \ref{t:main} below, we shall  
establish this criterion for a stationary $\sas$ random field 
indexed by an arbitrary amenable group. 

Thanks to \cite{dye:1965}, random fields indexed by amenable groups enjoy the following simpler and  useful characterization of weak mixing based on Definition~\ref{defn:F_mixing}. 

\begin{theorem}[\cite{bergelson:gorodnik:2004}, Theorem~1.6] \label{thm:wm_amenable}
	Let $G$ be an amenable group with an increasing F{\o}lner sequence $\mathcal{F}=\{F_n: n \geq 1\}$. Then a $G$-indexed stationary random field (not necessarily $\sas$) $\{Y_g\}_{g \in G}$ is weakly mixing if and only if it is $\mathcal{F}$-mixing. 
\end{theorem}

\begin{definition}
	Assume that $\mathcal{F}=\{F_n: n \geq 1\}$ is an increasing sequence of finite subsets $F_n \uparrow G$ (not necessarily satisfying the F{\o}lner condition). We say that a set $E \subset G$ has $\mathcal{F}$-density zero if 
	\[
	\lim_{n \to \infty} \frac{|E \cap F_n|}{|F_n|} =0.
	\]
\end{definition}

Then the proof of  \cite[Lemma~6.2]{petersen:1983} gives the following. (In the statement, `$g \to \infty,\,g \notin E$' simply means that $g$ stays away from $E$ and eventually escapes any finite subset.)

\begin{lemma}[Koopman-von Neumann] \label{lemma:Kv}
	Let $\psi: G \to [0,\infty)$ be a bounded function. Then $\displaystyle{\lim_{n \to \infty} \frac{1}{|F_n|} \sum_{g \in F_n} \psi(g) =0}$ if and only if there exists $E \subset G$ of $\mathcal{F}$-density zero such that $\displaystyle{\lim_{g \to \infty, g \notin E} \psi(g)=0}$. 
\end{lemma}

In order to use the ergodic theorem of \cite{tempelman:2015}, we shall further need to choose the F{\o}lner sequence to be tempered (in the sense of \cite{lindenstrauss:2001}) as defined below. 

\begin{defn}
	A sequence of finite sets $F_n \subset G$ is called tempered if there exists $C \in (0, \infty)$ such that for all $n \geq 2$, the Shulman condition
	\[
	\left|\bigcup_{k=1}^{n-1} F_k^{-1} F_n\right| \leq C |F_n|.
	\]
	holds.
\end{defn}
\noindent It was shown in \cite[Proposition~1.4, Theorem~1.2]{lindenstrauss:2001}  that every amenable group admits an increasing tempered F{\o}lner sequence along which pointwise ergodic theorem holds.  
In addition to the work of \cite{lindenstrauss:2001} (and its refinement \cite{tempelman:2015})  mentioned above, we also need an extension of a result from \cite{gross:1994}. The statement is essentially the countable groups version of   \cite[Equation~(3.1)]{samorodnitsky:2005a}.

\begin{lemma}\label{lemma_gross} Let $G$ be any countable group and  $\mathcal{F}=\{F_n: n \geq 1\}$ be a sequence of finite subsets $F_n \uparrow G$. Then a $G$-indexed stationary $\sas$ random field $\mathbf{Y}=\{Y_g\}_{g\in G}$, with a Rosi{\'n}ski representation of the form \eqref{eqthm}, is $\mathcal{F}$-mixing if and only if for all $\delta >0$ and for all $\epsilon > 0$, 
	\begin{equation}
		\lim_{n \to \infty} \frac{1}{|F_n|} \sum_{g \in F_n} \mu\big(x \in S:\delta\leq |f_e(x)| \leq \delta^{-1},\, |f_g(x)| \geq \epsilon\big) =0. \label{condn_gross}
	\end{equation}
	In particular, if $G$ is a countable amenable group and $\mathcal{F}=\{F_n: n \geq 1\}$ is an increasing F{\o}lner sequence, then \eqref{condn_gross} is equivalent to weak mixing of $\{Y_g\}_{g \in G}$. 
\end{lemma}

\begin{proof}
	The first part follows from  the techniques of \cite{wang:roy:stoev:2013} (see the proof of Proposition $4.2$ in the appendix therein) with the help of Lemma~\ref{lemma:Kv} above. The second part follows from the first one using Theorem~\ref{thm:wm_amenable}. 
\end{proof}

\section{Truncation and Maharam Extension}\label{sec-mptons}
In this section, we apply a probabilistic truncation argument along with Maharam skew product to establish that it is enough to check condition \eqref{condn_gross} for indicator functions $f$  and measure-preserving $G$-actions $\{\phi_g\}$ (see Theorem \ref{t:mpns} for a precise statement) in \eqref{eqthm}. Throughout, we assume that $G$ is any countably infinite group and  $\mathcal{F}=\{F_n: n \geq 1\}$ is a sequence of finite subsets $F_n \uparrow G$. The results of this section do not depend on the amenability of the group $G$ and hold for all countable groups.

Before we can move to the main theorem of this section, we shall need a series of lemmas. We assume that the stationary $\sas$ random field $\mathbf{Y}=\{Y_g\}_{g \in G}$ has an integral representation of the form 
\[Y_g \eqd \int_S f_g(x)M(dx), \mbox{ \ \  } g \in G,\]
where $M$ is an $\sas$ random measure on some $\sigma$-finite standard Borel space $(S,\mathcal{S}, \mu)$ and $f_g$ is given by \eqref{eqthm}, for some non-singular group action $\{\phi_g\}_{g \in G}$, a $\pm 1$-valued measurable cocycle $\{c_g\}_{g\in G}$ and a real-valued function $f = f_e \in \LL^\alpha (S,\mu)$. 

First observe that, without loss of generality, we can assume that $\mu$ is a probability measure. For if not, we can always get a probability measure $\nu$ equivalent to the sigma-finite measure $\mu$ and then the random field $\mathbf{Y}=\{Y_g\}_{g \in G}$ has an integral representation of the form 
\[Y_g \eqd \int_S h_g(x)N(dx), \mbox{ \ \  } g \in G,\]
where $N$ is an S$\alpha$S random measure on $(S,\mathcal{S}, \nu)$ and 
\[h(s)=f(s)\left(\frac{d\mu}{d\nu}(s)\right)^{1/\alpha}\in \LL^\alpha (S,\nu),\]
and $h_g$ is as given in \eqref{eqthm} with $h,\nu$ in place of $f,\mu$. Henceforth, we assume in this section that $\mu$ is a probability measure. Throughout, we fix $\delta>0$ and for simplicity of notation, we let
\[A:=\{x\in S:\delta\le|f_e(x)|\le \delta^{-1}\}\,. \]
Then, we have the following lemma that shows that it is enough to check \eqref{condn_gross} for $f$ bounded. 

\begin{lemma}\label{l:bd} Let $\mu$ be a probability measure. Then for any $\e>0$,
	\begin{eqnarray*}
		&&\limsup_{n } \frac{1}{|F_n|} \sum_{g \in F_n} \mu \left(x\in A:|f\circ \phi_g(x)|^\alpha\frac{d(\mu\circ\phi_g)}{d\mu}(x)>\e \right)\\
		&\le& \limsup_{L\to \infty}\limsup_{n \to \infty} \frac{1}{|F_n|} \sum_{g \in F_n} \mu \left(x\in A:|f\ind_{|f|\le L}\circ \phi_g(x)|^\alpha\frac{d(\mu\circ\phi_g)}{d\mu}(x)>\e/2 \right).
	\end{eqnarray*}
	In particular, if \eqref{condn_gross} holds for all $f$ bounded, then \eqref{condn_gross} holds for all $f\in \LL^{\alpha}(S,\mu)$. 
\end{lemma}

\begin{proof} For any fixed $L>0$, we can write $f=f\ind_{|f|\le L}+f\ind_{|f|>L}$. So, for any fixed $\epsilon>0$, 
	\begin{eqnarray*}
		\mu(x\in A:|f_g(x)|^\alpha>\e)&\le& \mu\left(x\in A:|f\ind_{|f|\le L}\circ \phi_g(x)|^\alpha\frac{d(\mu\circ\phi_g)}{d\mu}(x)>\e/2\right)\\
		&&+\mu\left(x\in A:|f\ind_{|f|> L}\circ \phi_g(x)|^\alpha\frac{d(\mu\circ\phi_g)}{d\mu}(x)>\e/2\right)
	\end{eqnarray*}
	Now, we can approximate the second term on the right-hand side as follows.
	\begin{eqnarray*}
		&&\mu\left(x\in A:|f\ind_{|f|> L}\circ \phi_g(x)|^\alpha\frac{d(\mu\circ\phi_g)}{d\mu}(x)>\e/2\right)\\
		&\le&\int\ind\left(|f\circ \phi_g(x)|> L,|f\circ \phi_g(x)|^\alpha\frac{d(\mu\circ\phi_g)}{d\mu}(x)>\e/2\right)d\mu\\
		&\le&\frac{2}{\e}\int|f\circ \phi_g(x)|^\alpha\frac{d(\mu\circ\phi_g)}{d\mu}(x)\ind\left(|f\circ \phi_g(x)|> L\right)d\mu\\
		&=&\frac{2}{\e}\int|f\circ \phi_g(x)|^\alpha\ind\left(|f\circ \phi_g(x)|> L\right)d\mu\circ\phi_g\\
		&=&\frac{2}{\e}\int|f(x)|^\alpha\ind\left(|f(x)|> L\right)d\mu.\end{eqnarray*}
	Thus, 
	\begin{eqnarray*}
		\sup_{n } \frac{1}{|F_n|} \sum_{g \in F_n} \mu \left(x\in A:|f\ind_{|f|> L}\circ \phi_g(x)|^\alpha\frac{d(\mu\circ\phi_g)}{d\mu}(x)>\e/2 \right)\\
		\le \frac{2}{\e}\int|f(x)|^\alpha\ind\left(|f(x)|> L\right)d\mu\rightarrow 0
	\end{eqnarray*}
	as $L\to \infty$ by the dominated convergence theorem, since $f\in \LL^{\alpha}(S,\mu)$. From this, the lemma follows. 
\end{proof}
Henceforth, we assume that $f$ is bounded, that is, there exists some $K>0$ such that $|f(x)|\le K$ for all $x\in S$. In the next lemma, we show that we can assume that the Radon-Nikodym derivatives are bounded away from $0$ and infinity.

\begin{lemma}\label{l:radbd} Let $\mu$ be a probability measure and $|f|\le K$ for some $K>0$. Then
	\begin{align}
		&\sup_{n } \frac{1}{|F_n|} \sum_{g \in F_n} \mu \bigg(x\in A:|f\circ \phi_g(x)|^\alpha\frac{d(\mu\circ\phi_g)}{d\mu}(x)>\e, \label{e:toshow} \\ &\;\;\;\;\;\;\;\;\;\;\;\;\;\;\;\;\;\;\;\;\;\;\;\;\;\;\;\;\;\;\;\;\;\;\;\;\;\;\;\;\;\;\;\;\;\;\;\;\;\;\;\;\;\;\;\;\;\;\;\;\;\;\;\;\;\;\frac{d(\mu\circ\phi_g)}{d\mu}(x)\not\in [L^{-1},L] \bigg)\to 0 \nonumber
	\end{align}
	as $L\to \infty$.
\end{lemma}
\begin{proof} When $\frac{d(\mu\circ\phi_g)}{d\mu}(x)$ is large, we have
	\begin{eqnarray*}
		&&\mu \left(x\in A:|f\circ \phi_g(x)|^\alpha\frac{d(\mu\circ\phi_g)}{d\mu}(x)>\e, \frac{d(\mu\circ\phi_g)}{d\mu}(x)>L \right)\\
		&\le& \mu \left( \frac{d(\mu\circ\phi_g)}{d\mu}(x)>L \right)\\
		&\le&\frac{1}{L}\int \frac{d(\mu\circ\phi_g)}{d\mu}(x) d\mu\\
		&=&\frac{1}{L}\int d\mu\circ\phi_g(x)=\frac{1}{L}\,,
	\end{eqnarray*}
	where we have used the Markov inequality in the third line and the last equality uses the fact that $\mu$ is a probability measure. On the other hand, by choosing $L$ large enough so that $L>\e^{-1}K^\alpha$, we have
	\begin{eqnarray*}
		&&\mu \left(x\in A:|f\circ \phi_g(x)|^\alpha\frac{d(\mu\circ\phi_g)}{d\mu}(x)>\e, \frac{d(\mu\circ\phi_g)}{d\mu}(x)<L^{-1} \right)\\
		&\le& \mu \left( \frac{d(\mu\circ\phi_g)}{d\mu}(x)>\e K^{-\alpha}, \frac{d(\mu\circ\phi_g)}{d\mu}(x)<L^{-1} \right)=0\,,\end{eqnarray*}
	where we used that $|f|\le K$. Thus, for $L>\e^{-1}K^\alpha$, the quantity in \eqref{e:toshow} is at most $L^{-1}$, which goes to $0$ as $L\to \infty$. This proves the lemma. 
\end{proof}

The next lemma converts the problem of checking condition \eqref{condn_gross} for non-singular actions to that for measure-preserving actions. First, we need to recall the Maharam extension (see, e.g., Chapter 3.4 of \cite{aaronson:1997}). For the non-singular map $\phi_g$ of $G$ acting on $(S,\mu)$, let
\[w_g(x):=\frac{d\mu\circ\phi_g}{d\mu}(x),\qquad g\in G, x\in S. \]
Then the group action $\phi_g^*$ of $G$ on $(S\times (0,\infty), \mu\otimes\mbox{Leb})$ defined, for each $g \in G$, as
\begin{equation}\label{e:Maharram}
	\phi_g^*(x,y):=\left(\phi_g(x),\frac{y}{w_g(x)}\right), \qquad x\in S, \,y>0
\end{equation}
preserves the measure $\mu\otimes\mbox{Leb}$; see \cite{maharam:1964}.

\begin{lemma}\label{l:final} Let $\mu$ be a probability measure and $|f|\le K$ for some $K>0$. For any $L>1$, define $h^{(L)}\in \LL^\alpha(S\times (0,\infty), \mu\otimes\mbox{Leb})$ as
	\[h^{(L)}(x,y):=\frac{|f(x)|}{y^{1/\alpha}}\ind\left(\frac{1}{2L}\le y\le 2L\right)\,, \qquad x\in S,\, y>0\,.\]
	Then for any $\e>0$,
	\begin{eqnarray}\label{e:toshow2}
		&&\limsup_{n } \frac{1}{|F_n|} \sum_{g \in F_n} \mu \left(x\in A:|f\circ \phi_g(x)|^\alpha\frac{d(\mu\circ\phi_g)}{d\mu}(x)>\e \right)\\
		&\le& \frac{2}{3}\limsup_{L\to \infty}\limsup_{n \to \infty} \frac{1}{|F_n|} \sum_{g \in F_n} \mu\otimes\mbox{Leb} \left((x,y)\in A':|h^{(L)}\circ\phi_g^*(x,y)|^\alpha>\e/2 \right)\,,\nonumber
	\end{eqnarray}
	where
	\[A':=\{(x,y)\in S\times(0,\infty):2^{-1}\delta\le|h^{(L)}(x,y)|\le 2\delta^{-1}\}\,. \]
\end{lemma}

\begin{proof}First observe that for any $L>1$,
	\begin{eqnarray*}
		&&\frac{3}{2}\mu \left(x\in A:|f\circ \phi_g(x)|^\alpha w_g(x)>\e, w_g(x)\in [L^{-1},L] \right)\\
		&=&\mu\otimes\mbox{Leb}\Big((x,y):\delta\le |f(x)|^\alpha\le \delta^{-1}, |f\circ\phi_g(x)|^\alpha w_g(x)>\e, w_g(x)\in [L^{-1},L],  \\
		&&\qquad\qquad\qquad\qquad\qquad \qquad\qquad\qquad\qquad\qquad \qquad\qquad\qquad \quad y\in [2^{-1},2]\Big) \\
		&\le&\mu\otimes\mbox{Leb}\Big((x,y):2^{-1}\delta\le \frac{|f(x)|^\alpha}{y}\le 2\delta^{-1}, \, \frac{|f\circ\phi_g(x)|^\alpha}{y} w_g(x)>\e/2, \\
		&&\qquad\qquad\qquad\qquad\qquad \qquad y\in [(2L)^{-1},2L], \,y\in [(2L)^{-1}w_g(x),2Lw_g(x)]\Big)  \\
		&=&\mu\otimes\mbox{Leb}\Big((x,y):2^{-1}\delta\le |h^{(L)}(x,y)|\le 2\delta^{-1}, \,|h^{(L)}\circ\phi_g^*(x,y)|^\alpha>\e/2\Big)\,.
	\end{eqnarray*}
	This, together with Lemma \ref{l:radbd}, proves \eqref{e:toshow2}.
\end{proof}

Now, we are ready to prove the main theorem of this section.
\begin{theorem}\label{t:mpns} Let $\{\phi_g\}$ be any non-singular group action of $G$ on $(S,\mu)$. Assume that for the measure-preserving group action $\{\phi_g^*\}$ of $G$ on $(S\times (0,\infty), \mu\otimes\mbox{Leb})$ defined in \eqref{e:Maharram}, and any set $E\subseteq S\times(0,\infty)$ with $\mu\otimes \mbox{Leb}(E)<\infty$, we have 
	\begin{equation*}\lim_{n \to \infty} \frac{1}{|F_n|} \sum_{g \in F_n} \mu\otimes\mbox{Leb}\big(E\cap (\phi_g^{*})^{-1}(E)) =0.
	\end{equation*}
	Then for any $f\in \LL^\alpha(S,\mu)$ and all $\e>0,\,\delta>0$, condition \eqref{condn_gross} holds, that is, 
	\begin{equation}\label{e:thmeq}
		\lim_{n \to \infty} \frac{1}{|F_n|} \sum_{g \in F_n} \mu\big(x \in S:\delta\leq |f_e(x)| \leq \delta^{-1},\, |f_g(x)| \geq \epsilon\big) =0.
	\end{equation}
\end{theorem}
\begin{proof} Fix $\delta>0,\, \e>0$ and any non-singular $G$-action $\{\phi_g\}$ on $(S,\mu)$ and any $f\in \LL^\alpha(S,\mu)$. By Lemma \ref{l:bd} and the discussion preceding it, we can assume that $\mu$ is a probability measure and $f$ is bounded, that is, $|f|\le K$ for some $K>0$. Then by Lemma \ref{l:final}, \eqref{e:thmeq} holds if for all $L>1$ and for $\phi_g^*$ defined from $\phi_g$ according to $\eqref{e:Maharram}$, we can show that
	\begin{equation}\label{e:thmtoshow}
		\lim_{n \to \infty} \frac{1}{|F_n|} \sum_{g \in F_n} \mu\otimes\mbox{Leb} \left((x,y)\in A':|h^{(L)}\circ\phi_g^*(x,y)|^\alpha>\e/2 \right)=0\,,
	\end{equation}
	where \[h^{(L)}(x,y):=\frac{|f(x)|}{y^{1/\alpha}}\ind\left(\frac{1}{2L}\le y\le 2L\right)\qquad \mbox{ and}\] 
	\[A':=\{(x,y)\in S\times(0,\infty):2^{-1}\delta\le|h^{(L)}(x,y)|\le 2\delta^{-1}\}.\]
	Fix $L>1$ and note that
	\[|h^{(L)}(x,y)|=\frac{|f(x)|}{y^{1/\alpha}}\ind\left(\frac{1}{2L}\le y\le 2L\right)\le K(2L)^{1/\alpha}\ind\left(\frac{1}{2L}\le y\le 2L\right)\,.\]
	Define the set $E_L$ as
	\[E_L:=S\times [(2L)^{-1},2L]\,.\]
	Then $\mu\otimes \mbox{Leb}(E_L)=\mbox{Leb}([(2L)^{-1},2L])<\infty$ and
	\begin{eqnarray*}
		&&\mu\otimes\mbox{Leb} \left((x,y)\in A':|h^{(L)}\circ\phi_g^*(x,y)|^\alpha>\e/2 \right)\\
		&\le& \mu\otimes\mbox{Leb}\left((x,y)\in S\times(0,\infty):\ind_{E_L}(x,y)>0, \,\ind_{E_L}\circ\phi_g^*(x,y)>0\right)\\
		&=& \mu\otimes\mbox{Leb}\big(E_L\cap (\phi_g^{*})^{-1}(E_L))\,.
	\end{eqnarray*}
	As $\mu\otimes \mbox{Leb}(E_L)<\infty$, by the assumption of the theorem,
	\[\lim_{n \to \infty} \frac{1}{|F_n|} \sum_{g \in F_n} \mu\otimes\mbox{Leb}\big(E_L\cap (\phi_g^{*})^{-1}(E_L)) =0.\]
	Hence, \eqref{e:thmtoshow} holds for all $L>1$, thus implying \eqref{e:thmeq}. This proves the theorem.
\end{proof}

\section{Stable Random fields indexed by amenable groups}
\subsection{Criterion for Weak Mixing}
In this subsection, we prove one of the main results of this paper that characterizes weak mixing of a stationary $\sas$ random field indexed by any countable amenable group $G$. In a Rosi{\'n}ski representation of a stationary $\sas$ random field indexed by $G$ as in Theorem \ref{omni-sas}, if the group action $\{\phi_g\}_{g \in G}$ has only the null (resp.\ positive) part, that is, $\mathcal{P}=\emptyset$ modulo $\mu$ (resp.\ $\mathcal{N}=\emptyset$) in the Neveu decomposition, then we say that the random field is generated by a null (resp.\ positive) action in that representation. The following proposition shows that this does not depend on the particular Rosi\'nski representation chosen as long as the group $G$ is amenable. 

\begin{prop} If a stationary $\sas$ random field indexed by any countable amenable group $G$ is generated by a positive (resp. null) action in one Rosi\'nski representation, then in every other Rosi\'nski representation it will be generated by a positive (resp. null) action. 
\end{prop}

\begin{proof}
	This  follows by mimicking the proof of   \cite[Theorem~3.1]{wang:roy:stoev:2013}, which in turn essentially relies on the work of \cite{takahashi:1971}. The latter work deals with amenable semigroups of contractions. Hence, it  directly applies to our situation. This yields a characterization for a stationary $\sas$ random field indexed by a discrete amenable group $G$ to be generated by a null (or a positive) $G$-action. 
\end{proof}

\begin{rmk} \label{rmk:pos_null_decomp} \textnormal{Following closely the proof of  \cite[Corollary~3.3]{wang:roy:stoev:2013}, we can also establish that for any   stationary $\sas$ field $\mathbf{Y}=\{Y_g\}_{g \in G}$ indexed
		by an amenable group, the decomposition into null and positive parts does not depend on the choice of Rosi\'nski representation and hence is unique in law. We can, therefore, define these to be the positive and null parts, respectively, of $\mathbf{Y}$.}
\end{rmk}

Following is the main result of this section. It extends significantly the main theorems of \cite{samorodnitsky:2005a, wang:roy:stoev:2013} leading to a crucial improvement (see Theorem~\ref{thm:wm_charac_vNalg} below) of the results of \cite{roy:2020}. 

\begin{theorem}\label{t:main} Any stationary $\sas$ random field indexed by a countably infinite amenable group $G$ is weakly mixing if and only if it is generated by a null action (in some and hence all Rosi{\'n}ski representations).
\end{theorem}

For the proof of this theorem, we will need some definitions, which we recall here from Tempelman \cite{tempelman:2015}. A mapping $T: g\mapsto T_g$ on $G$ is said to be a \emph{right $\LL^p$-representation} of $G$ if, for each $g\in G$, $T_g$ is an operator in $\LL^p(\Lambda,\mathcal{L}, m)$ (here $(\Lambda,\mathcal{L}, m)$ is a $\sigma$-finite standard measure space), $T_{g_1g_2}=T_{g_2}T_{g_1}, \, g_1,g_2\in G$ and $T_e$ is the identity operator. A representation $T$ is \emph{bounded} if $\|T\|_p:=\sup_{g\in G}\|T_g\|_p<\infty$.  For the $T_g$-image of $f\in \LL^p(\Lambda, m)$, we use the notation $fT_g$. We say that a bounded representation $T$ is a \emph{Lamperti} $\LL^p$-representation if $T_g$ is Lamperti for all $g\in G$, that is, for all $f,h\in \LL^p$ with $fh=0$ $m$-almost everywhere (a.e.), one has $(fT_g)(hT_g)=0$  $m$-a.e. Now, we are ready to recall the following result from \cite{tempelman:2015} for a countable amenable group $G$.
\begin{theorem}\cite[Theorem 3.9]{tempelman:2015}\label{t:temp} Let $\{F_n\}$ be a tempered F{\o}lner sequence in $G$. If $1<p<\infty$ and $T: g\mapsto T_g$ is a Lamperti $L^p$-representation, then for each $f\in \LL^p(\Lambda,\mathcal{L}, m)$, 
	\[\lim_{n\to \infty} \frac{1}{|F_n|}\sum_{g\in F_n}fT_g=h \quad m \mbox{-a.e.}\,,\]
	where $h$ is $T_g$-invariant.
\end{theorem}
\noindent The $h$ in the above limit is the \emph{mean value} of the orbit $\{fT_g, g\in G\}$ and is the projection of $f$ onto the space of all $T$-invariant functions in $\LL^p$.

Now, let us fix any measure-preserving group action $\phi_g^*$ of $G$ acting on $(\Lambda,\mathcal{L}, m)$. We define for each $g\in G$, $T_g: \LL^2(\Lambda,\mathcal{L}, m)\mapsto \LL^2(\Lambda,\mathcal{L}, m)$ as $fT_g=f\circ \phi_g^*$. Note that this is bounded as
\[\|fT_g\|_2^2=\int |f\circ \phi_g^*|^2dm=\int |f|^2dm=\|f\|_2^2\]
since $\phi_g^*$ is measure $m$-preserving. Thus, $\|T_g\|_2=1$ for all $g\in G$, so that $\|T\|_2:=\sup_{g\in G}\|T_g\|_2=1<\infty$. Also, this representation is Lamperti because for all $f,h\in \LL^2(\Lambda, m)$ with $fh=0$ $m$-a.e. and all $g\in G$,
\[(fT_g)(hT_g)=(f\circ \phi_g^*)(h\circ \phi_g^*)=0 \quad m\mbox{-a.e.}\,.\]
Finally, we are ready to prove Theorem \ref{t:main}.

\begin{proof}[Proof of Theorem \ref{t:main}] {We first note the easy direction: if the random field $\mathbf Y$ is not generated by a null action, then it is not  weakly mixing. To see this, we can first without loss of generality assume that the action of $G$  is positive (because of the positive-null decomposition of $\mathbf Y$ as in Remark~\ref{rmk:pos_null_decomp} and the criterion in Lemma ~\ref{lemma_gross}).  Thus, $\mathbf Y$ has an integral representation on $(S,\nu)$ where $\nu$ is a probability measure, which is preserved by the action of the amenable group $G$. Consider a sequence of Folner sets $F_n$ for $G$. Then for any set $A$ with $\nu(A) > 0$, $\frac{1}{|F_n|}\sum_{g \in F_n} \ind_{g.A}$ converges by the ergodic theorem of Lindenstrauss for amenable groups ~\cite{lindenstrauss:2001} to some invariant $h$ such that $\int h d\nu>0$ (since by the dominated convergence theorem, this equals $\nu(A)$). Hence, again by the dominated convergence theorem, $1/|F_n| \int_S \ind_A  \sum_{g \in F_n} \ind_{g.A}d\nu$ converges to $\int_Ah d\nu$, which we claim is a positive number; this in turn violates the weak mixing criterion in Lemma ~\ref{lemma_gross} for $f_e=\ind_A$. To see the claim note that, if $\int_Ah d\nu=0$, then $\int_{g.A}h d\nu=0$ for all $g\in G$, and hence $\int h d\nu=0$. From here, it is easy to generalize to any $f_e \in \LL^\alpha (S,\nu)$ using standard arguments.}

	Now we assume that the random field is generated by a null non-singular action $\phi_g$ acting on $(S,\mu)$. Define the measure-preserving action $\phi_g^*$ of $G$ on $(\Lambda:=S\times (0,\infty), \, \mathcal{L}:= \mathcal{S} \otimes \mathcal{B}_{(0,\infty)}, \, m:=\mu\otimes\mbox{Leb})$ as defined in \eqref{e:Maharram}. We first claim that $\phi_g^*$ is also null. This is because if $B\subseteq S$ is a weakly wandering set (see Section~\ref{sec-wm}) for $\phi_g$, then $B\times (0,\infty)$ is a weakly wandering set for $\phi_g^*$. Therefore, it follows that $\Lambda:=S\times (0,\infty)$ is the measurable union of its weakly wandering subsets since $S$ is so. 
	
	For $f\in \LL^2(\Lambda, m)$, the operator $fT_g=f\circ \phi_g^*$ is a bounded Lamperti $\LL^2$-representation by the discussion following Theorem \ref{t:temp}. Fix any set $E\subseteq S\times(0,\infty)$ with $m(E)=\mu\otimes \mbox{Leb}(E)<\infty$, and let $f=\ind_E$. Clearly $f\in \LL^2(\Lambda, m)$ and by Theorem \ref{t:temp} we have
	\[\lim_{n\to \infty} \frac{1}{|F_n|}\sum_{g\in F_n}\ind_{(\phi_g^*)^{-1}(E)}=h \quad m \mbox{-a.e.}\]
	for some $h$ which is $\phi_g^*$ invariant. Clearly $h\geq 0$ and by the Fatou's Lemma,
	\begin{eqnarray*}
		\int_\Lambda h dm\le \liminf_{n\to \infty}\int_\Lambda \frac{1}{|F_n|}\sum_{g\in F_n}\ind_{(\phi_g^*)^{-1}(E)}dm=\liminf_{n\to \infty} \frac{1}{|F_n|}\sum_{g\in F_n}\int_\Lambda\ind_{(\phi_g^*)^{-1}(E)}dm\\
		=\int_\Lambda \ind_{E}dm=m(E)<\infty\,,
	\end{eqnarray*}
	using the fact that $\phi_g^*$ preserves the measure $m$. Thus $h\in \LL^1(\Lambda,m)$. It follows that $d\nu=hdm$ is a finite $\phi_g^*$-invariant measure equivalent to $m$ on $\{h>0\}$. Since the action $\phi_g^*$ is null, it follows that $h=0$ $m$-a.e. Hence
	\[\lim_{n\to \infty} \frac{1}{|F_n|}\sum_{g\in F_n}\ind_{(\phi_g^*)^{-1}(E)}=0 \quad m \mbox{-a.e.}\,.\]
	Multiplying both sides by $\ind_E$, we have
	\[\lim_{n\to \infty} \frac{1}{|F_n|}\sum_{g\in F_n}\ind_{E\cap(\phi_g^*)^{-1}(E)}=0 \quad m \mbox{-a.e.}\,.\]
	Moreover, for all $n$, 
	\[0\le \frac{1}{|F_n|}\sum_{g\in F_n}\ind_{E\cap(\phi_g^*)^{-1}(E)}\le \ind_E\]
	and $\int_\Lambda \ind_E dm=m(E)<\infty$. Hence, by the dominated convergence theorem,
	\[\lim_{n\to \infty} \frac{1}{|F_n|}\sum_{g\in F_n}m\left(E\cap(\phi_g^*)^{-1}(E)\right)=0 \,.\]
	Thus by Theorem \ref{t:mpns}, condition \eqref{condn_gross} holds for any $f\in \LL^\alpha(S,\mu)$ and all $\e>0,\,\delta>0$. Hence by Lemma \ref{lemma_gross}, the $\sas$ random field is weakly mixing since $G$ is amenable.
\end{proof}

\subsection{Connections to von Neumann Algebras} \label{sec:vNalg}
In this subsection, we give a  characterization of weak mixing of a stationary $\sas$ random field indexed by a countable amenable group $G$ in terms of von Neumann algebras. We refer to
\cite{sunder:1987} for basics of von Neumann algebras and factors. 
Recall that a von Neumann algebra $M \subseteq \mathcal{B}(\mathcal{H})$ (for some separable Hilbert space $\mathcal{H}$) is called a factor if its center $Z(M):=M \cap M^\prime$ is trivial, i.e., $Z(M)= \mathbb{C} 1$. We shall, in general, need  von Neumann's    \emph{central  decomposition}
\begin{equation}
	M = \int_Y M_y\, \nu(dy)  \label{decomp:central}
\end{equation}
of $M$ as a direct integral of factors $\{M_y: y \in Y\}$ over a $\sigma$-finite measure space $(Y, \mathcal{Y}, \nu)$.
Recall (see \cite{sunder:1987}) that a  factor  is of type $II_1$ if it is infinite-dimensional and it admits a normalized trace.
We say \cite[Definition~3.6]{roy:2020} that a von Neumann algebra $M$ does not admit a $II_1$ factor in its central decomposition \eqref{decomp:central} if for $\nu$-almost all $y \in Y$, $M_y$ is a not a $II_1$ factor.

We now describe a crossed product von Neumann algebra called the \emph{group measure space construction} \cite{murray:vonneuman:1936}.
This encodes ergodic  properties of a  non-singular $G$-action $\{\phi_g\}$. Any such action induces, for each $g \in G$, the isometry $\pi_g: \LLL^2(S, \mu) \to \LLL^2(S, \mu)$ given by
\[
(\pi_g h)(s) = h \circ \phi_g(s) \left(\frac{d \mu \circ \phi_g}{d\mu}(s)\right)^{1/2}, \; s \in S.
\]
The unitary representation $\{\pi_g\}_{g \in G}$ of $G$ on $\LLL^2(S, \mu)$  is called the \emph{Koopman representation.} For all $a \in \LLL^\infty(S, \mu)$ (thought of as acting on $\LLL^2(S, \mu)$ by multiplication), for all $g \in G$ and for all $h \in \LLL^2(S, \mu)$,
\begin{equation}
	(\pi_g \, a \, \pi_{g^{-1}} h)(s) = ((\sigma_g a) h)(s), \;\; s \in S, \label{eqn:crossed_product}
\end{equation}
where $\sigma_g a:= a \circ \phi_g$. Thus, the Koopman representation ``normalizes'' $\LLL^\infty(S, \mu)$ inside $\mathcal{B}(\LLL^2(S, \mu))$.


Consider the von Neumann algebra
$
\mathcal{B}(l^2(G) \otimes \LLL^2(S, \mu)) = \overline{\mathcal{B}(l^2(G)) \otimes \mathcal{B}(\LLL^2(S, \mu))}
$
(where the closure is with respect to the weak/strong operator topology). Define a representation of $G$ by $g \mapsto u_g:= \lambda_g \otimes \pi_g$, where $\{\lambda_g\}$ is the left regular representation and $\{\pi_g\}$ is the Koopman representation. We also represent $\LLL^\infty(S, \mu)$ by $a \mapsto 1 \otimes \mathcal{M}_a$, where $\mathcal{M}_a$ is the multiplication (by $a$) operator on $\LLL^2(S, \mu)$. Then,
\begin{equation}
	u_t (1 \otimes \mathcal{M}_a) u_{t^{-1}} = 1 \otimes \mathcal{M}_{\sigma_t a}\,.  \label{eqn:crossed_product_int}
\end{equation}
Define the \emph{group measure space construction} as the double commutant
\[
\LLL^\infty(S, \mu) \rtimes G := \{u_t, 1 \otimes \mathcal{M}_a: t \in G, \, a \in \LLL^\infty(S, \mu)\}^{\prime\prime}.
\]
It was shown in \cite[Theorem~5.4]{roy:2020}  that for any countable group $G$,  non-singular $G$-actions coming from two minimal (and hence Rosi\'nski) representations of a stationary $\sas$ random field $\mathbf{Y}=\{Y_g\}_{g \in G}$ are $W^\ast$-equivalent, i.e.\ the corresponding group measure space constructions are isomorphic as von Neumann algebras. This algebra was defined as the \emph{minimal group measure space construction} of $\mathbf{Y}$.

Two stationary $\sas$ fields $\big\{X^{(1)}_g\big\}_{g \in G_1}$ and $\big\{X^{(2)}_g\big\}_{g \in G_2}$ indexed by two (possibly non-isomorphic) countable groups are called $W^\ast_R$-equivalent (resp., $W^\ast_m$-equivalent) if the group measure space construction corresponding to a Rosi\'{n}ski (resp., minimal) representation of  $\big\{X^{(1)}_g\big\}_{g \in G_1}$ is isomorphic (as a von Neumann algebra) to the group measure space construction corresponding to a Rosi\'{n}ski (resp., minimal) representation of $\big\{X^{(2)}_g\big\}_{g \in G_2}$. A property $P$ of stationary $\sas$ random fields (indexed by a class $\mathcal{G}$ of countable groups) is called $W^\ast_R$-rigid (resp., $W^\ast_m$-rigid) for $\mathcal{G}$ if whenever two such fields (not necessarily indexed by the same group) are $W^\ast_R$-equivalent (resp., $W^\ast_m$-equivalent), one enjoys property $P$ if and only if the other one does. By   \cite[Theorem~3.1]{rosinski:1995}, any minimal representation is a Rosi\'{n}ski representation (but the converse does not hold). Hence  $W^\ast_m\mbox{-equivalence}$ implies $W^\ast_R\mbox{-equivalence}$ and $W^\ast_R$-rigidity implies $W^\ast_m$-rigidity.

In \cite{roy:2020}, it was shown that weak mixing is a $W^\ast_R$-rigid (and hence $W^\ast_m$-rigid) property for stable random fields indexed by $\mathcal{G}:=\{\mathbb{Z}^d: d \in \mathbb{N}\}$ by showing that such a random field is weakly mixing if and only if the group measure space construction corresponding to some (equivalently, any) Rosi\'{n}ski representation does not admit a $II_1$ factor in its central decomposition. The same proof now combines with Theorem \ref{t:main} to give the following, resolving  completely
Problem~3 (and hence Problem~1 and Conjecture~2) of \cite{roy:2020}.

\begin{theorem} \label{thm:wm_charac_vNalg}
	Suppose that $G$ is a countable  amenable group and $\mathbf{Y}=\{Y_g\}_{g \in G}$ is a stationary $\sas$ random field. Suppose that $\mathbf{Y}$ admits a Rosi\'{n}ski representation such that the underlying non-singular $G$-action $\{\phi_g\}$ is free. Then the following are equivalent.
	\begin{enumerate}
		\item[(1)] $\{Y_g\}_{g \in G}$ is weakly mixing;
		\item[(2)]  the group measure space construction corresponding to $\{\phi_g\}_{g \in G}$ does not admit a $II_1$ factor in its central decomposition.
	\end{enumerate}
	In particular, weak mixing is a $W^\ast_R$-rigid (and hence $W^\ast_m$-rigid) property for the class $\mathcal{G}$ of  countably infinite amenable groups.
\end{theorem}

\begin{proof}  Because of Theorem~\ref{t:main} above, it is enough to establish that $\{\phi_g\}_{g \in G}$ is a null action if and only if the associated group measure space construction does not admit a $II_1$ factor in its central decomposition. This   follows verbatim the proof of equivalence of (3) and (4) of   \cite[Theorem~5.10]{roy:2020}. On the other hand, $W^\ast_R$-rigidity (and hence $W^\ast_m$-rigidity) of weak mixing follows directly from the characterization (2) above. This completes the proof.
\end{proof}

Following the argument used in the proof of Theorem~5.20 in \cite{roy:2020}, one can establish that having a nontrivial \emph{weakly mixing part} (i.e., \emph{null part} as in Remark~\ref{rmk:pos_null_decomp}) is also a $W^\ast_R$-rigid (and hence $W^\ast_m$-rigid) property for the class $\mathcal{G}$ of  countably infinite amenable groups.

The canonical action of any finitely generated group $G$ on its Furstenberg-Poisson boundary $(\partial G, \mu^{FP})$ is free and ergodic, and the associated group measure space construction is a type $III$ factor. This shows, thanks to Theorem~\ref{thm:wm_charac_vNalg}, that the corresponding stationary $\sas$ random fields are nontrivial and weakly mixing as long as $G$ is amenable and $\partial G$ is nontrivial. Examples arise from amenable lamplighter groups \cite{kaimanovich:vershik:1983}.

\section{Boundary actions of hyperbolic groups and Stable Fields}\label{sec-hyp}
We extract from Theorem \ref{t:mpns}, two sufficient conditions to prove ergodicity of a $G$-indexed $\sas$ random field. Suppose that we are given a 
non-singular action $g\to \phi_g$ of a discrete group $G$  on  a sigma-finite measure space $(Y,\mu)$. Let $(X,\nu) = (Y \times (0,\infty), \mu\otimes\mbox{Leb})$ denote the
Maharam extension (Equation \ref{e:Maharram}) equipped with a \emph{$\nu$- preserving action} $\phi_g^*$ as described before Lemma \ref{l:final}. There are two ways in which we shall apply Theorem \ref{t:mpns}
in this section.\\

\noindent {\bf Case 1:  $\mu(Y) < \infty$ and the $G$-action on $Y$ is non-singular with no restriction on Radon-Nikodym derivatives:}\\ Without loss of generality, we may assume that $\mu(Y)=1$, i.e.\ $Y$ is a probability measure space. In this case, any $E\subset X$ satisfying the condition 
$ \nu (E) < \infty$ in Theorem \ref{t:mpns} is contained (up to measure zero) in a countable \emph{disjoint}
union of sets $B$ of the form $B(a_i, b_i) =  Y \times (a_i, b_i)$ with $0<a_i <b_i <\infty$.
We shall show in Section \ref{sec-hypbdy} that for a number of actions with hyperbolic properties, the following condition holds for a non-singular action of a group $G$
on $X$.

\begin{cond}\label{cond-suff}
	$\frac{1}{|F_n|}
	\sum_{g \in F_n}  \nu\bigg(\left(Y \times (\frac{1}{K}, K)\right)\, \cap \,(\phi_g^*)^{-1}\left(Y \times (\frac{1}{K}, K)\right)\bigg)  \to 0 \,\ {\rm as} \,\ n \to \infty. $
\end{cond}

By Theorem \ref{t:mpns} this will suffice to prove ergodicity of the corresponding $G$-indexed $\sas$ random field.
Note further that if \[(y, t) \in A_K:=\left(Y \times \left(\frac{1}{K}, K \right)\right)\, \cap\, (\phi_g^*)^{-1}\left(Y \times \left(\frac{1}{K}, K \right)\right),\] by the definition of $\phi_g^*$, the Radon-Nikodym derivative satisfies $$\frac{1}{K^2}<\frac{d\mu\circ\phi_g^{-1}}{d \mu}(y)<K^2.$$ Hence,  $\nu(A_K)$
in Condition \ref{cond-suff} is at most $(K-K^{-1})\mu(B_K)$, where $B_K$ is given by $$B_K:=\bigg\{y \in  Y: \frac{1}{K^2}<\frac{d\mu\circ\phi_g^{-1}}{d \mu}(y)<K^2\bigg\}.$$

\noindent {\bf Case 2:  $\mu(Y) = \infty$ and the $G$-action on $Y$ is non-singular with  Radon-Nikodym derivatives uniformly bounded away from both $0$ and $\infty$:}\\
In this case, it follows immediately from Theorem \ref{t:mpns} that to prove ergodicity of the corresponding $G$-indexed $\sas$ random field, it suffices to check the following for $\mu(B)<\infty$:

\begin{cond}\label{cond-suff2}
	$\frac{1}{|F_n|}
	\sum_{g \in F_n} \bigg( \mu\big(B \cap \phi_g(B)\big) \bigg) \to 0 \,\  {\rm as} \,\  n \to \infty. $
\end{cond}
We shall check Condition \ref{cond-suff2} in Section \ref{sec-db} below.
\subsection{Boundary action of hyperbolic groups}\label{sec-hypbdy} We refer the reader to \cite[Section 3.1]{athreya:mj:roy:2019} for the relevant material on Patterson-Sullivan measures that we need in this subsection, and to the original sources
\cite{patterson-acta,sullivan-pihes,patterson-expo,coornert-pjm} for details.
We summarize what we need here from Patterson-Sullivan theory. Let $X$ be a Gromov hyperbolic metric space. Let $\partial_G X$ denote the Gromov boundary \cite{gromov-hypgps}. For $x, y\in X$ and $\xi \in \partial_G X$, $\beta_\xi(y,x)$ will denote the \emph{Busemann function} (see \cite{coornert-pjm} for details). In all applications below, $x$ will be chosen as an origin of the hyperbolic space $X$. 

\begin{defn}	\label{def-qcm}
	\cite[p. 721]{cm-gafa} 
	For $X$  Gromov-hyperbolic,   let $M(\partial_G X)$ denote the collection  of positive
	finite Borel measures on  $\partial_G X $. A  $G-$equivariant map $X \to M(\partial_G X)$ sending $x \to \mu_x$
	is said to be a $C-$quasiconformal
	density
	of dimension
	$v$ ($v\geq 0$),
	for some $C\geq 1$, if
	\begin{equation}\label{eq-qc}
		\frac{1}{C}exp(-v\beta_\xi(y,x))
		\leq
		\frac{d\mu_{x}}{d\mu_y} (\xi)
		\leq C exp(-v\beta_\xi(y,x))
	\end{equation}
	for all $x, y \in X$, $\xi\in \partial X$; in particular,
	$$\frac{1}{C} exp(-v\beta_\xi(o,g.o))\leq \frac{d\mu_{g.o}}{d\mu_o} (\xi)
	\leq C exp(-v\beta_\xi(o,g.o)),$$ for all $g \in G$.
\end{defn}	

\begin{eg}\label{eg-confden}
	Examples of quasiconformal densities arise in a hyperbolic setting from the following:
	\begin{enumerate}
		\item Patterson-Sullivan measures $\mups$ on the boundary of non-elementary hyperbolic groups, see \cite{coornert-pjm}.
		\item More generally, Patterson-Sullivan measures $\mups$ on limit sets $\Lambda_G$  of non-elementary groups $G$ acting properly by isometries on a hyperbolic space $X$. In this case, without loss of generality, we may choose $X$ to be the weak convex hull of $\Lambda_G \subset \partial_G X$, i.e.\ $X$ may be chosen to be the union of all bi-infinite geodesics joining pairs of points in $\Lambda_G$. Thus, $\Lambda_G=\partial_G X$ in this case. See \cite{coornert-pjm,sullivan-pihes}.
		\item A non-elementary hyperbolic group $G$ acting on its Poisson boundary \\$(\partial_G G,\mu_p)$, where $\mu_p$ denotes the hitting measure of random walks on the Poisson boundary. By a result of Kaimanovich \cite{kaimanovich}, the underlying space of the Poisson boundary in this case may be taken to be the Gromov boundary $\partial_G G$. Further, the natural metric in this context is the  Green metric \cite{bhm}--a metric quasi-isometric to the word metric on $G$. See \cite{kaimanovich,bhm}.
		\item Anosov representations of hyperbolic groups $G$ in semi-simple Lie groups $L$ (of possibly higher rank) furnish Patterson-Sullivan measures on their limit sets $\Lambda_G$ contained in $L/P$, where $P$ is a parabolic subgroup. 
		Fundamental work of Labourie \cite{labourie} and Guichard-Wienhard \cite{gw} shows that the limit set  $\Lambda_G$ is homeomorphic to the Gromov boundary $\partial_G G$ in this case.
		Dey and Kapovich \cite{dey-kap} prove that the Patterson-Sullivan measure turns out to be a conformal density with certain uniqueness properties in this case. See  \cite{dey-kap}. 
	\end{enumerate}
\end{eg}

Given a quasiconformal  $\mu$ on $\partial_G X$, a basepoint $o \in X$, and $g \in G$, we shall be interested in the $\mu-$measure of the set $$A(K,o,g):=\{\xi\in \partial_G X: |\beta_\xi(o,g.o)|\leq K \}.$$ Equivalently,
$A(K,o,g)$ is the set of points $\xi$ in $\partial_G X$ where $$\frac{1}{C} \leq \frac{d\mu_{o}\circ \phi_g^{-1}}{d\mu_o} (\xi) \leq C,$$ and the relationship between $K, C$ is obtained from Definition \ref{def-qcm} (essentially $C$ is an exponential of $K$).

\begin{lemma}\label{lem-mua20}
	Let $G, \partial_G X$ be any of the examples in Example \ref{eg-confden}, and $\mu$ be a conformal density as in Example \ref{eg-confden}.
	Then  $\mu(A(K,o,g)) \to 0$ as $g \to \infty$. (Here, we interpret $g \to \infty$ as the sequence $g.o$ exiting any bounded set.)
\end{lemma}

\begin{proof}
	For any geodesic segment $\lambda=[a,b]$ in a hyperbolic space $X$, let $\pi_\lambda$ denote a nearest point projection from $X$ onto $\lambda$.
	Then $\pi_\lambda$ extends to a (coarsely well-defined) map $\Pi_\lambda:
	\partial_G X \to \lambda$. For $o \in X$ as above, and $g \in G$, let $\lambda_g$ denote a geodesic joining $o, g.o$. Let $\Pi_g:
	\partial_G X \to \lambda_g$ denote the resulting nearest point projection 
	$\partial_G X$  to $\lambda_g$.
	
	Let  $m_g$ denote the mid-point of $\lambda_g$, and $[m_g-a,m_g+a]$ denote the geodesic subsegment of $\lambda_g$ of length $2a$ centered at $m_g$.
	We next observe that in each of the examples in Example \ref{eg-confden},  there exists $D$ depending on $K$ such that 
	$A(K,o,g)$ is contained  in $\Pi_g^{-1}([m_g-D,m_g+D])$.
	
	Hence $\mu(A(K,o,g)) \leq e^{-{v} (d(o,m_g)-D)}$. In particular, if $g \to \infty$,  $\mu(A(K,o,g)) \to 0$.
\end{proof}
\begin{theorem}\label{thm-hyp}
	In all the  examples in Example \ref{eg-confden}, Condition \ref{cond-suff} holds. Hence,  the associated $G-$indexed $\sas$ random field  is ergodic.
\end{theorem}

\begin{proof}
	For any $\ep>0$, Lemma \ref{lem-mua20} implies that there exists $N > 0$ such that
	$\mu(A(K,o,g)) < \ep$, whenever $d(o,g.o) \geq N$. In Condition \ref{cond-suff},
	define $F_m$ to be  $$F_m=\{g \in G: g(o,g.o) \leq m\},$$ though any other compact exhaustion would work as well.
	Next, given $r \geq 1$ choose $m>N$ such that $|F_M| > r |F_N|$  for all $M \geq m$.
	Finally, as pointed out in the discussion in the beginning of this section, it suffices to consider
	sets in Condition \ref{cond-suff} to be equal to $(\partial_G X) \times [\frac{1}{K},K]$ for some $K>1$, so that the measure of
	$B_K$ equals $\mu(A(K',o,g))$ (for some $K'$ depending on $K$).
	Then\\
	$$\sum_{g \in F_m} \big( \mu(B_K) \big)=\sum_{g \in F_m} \big( \mu(A(K',o,g)) \big) < |F_m|\ep + |F_N|.$$
	Hence, 
	$$\frac{1}{|F_m|}\sum_{g \in F_m} \big( \mu(B_K) \big)
	< \ep + \frac{1}{r}.$$
	Since  $\ep$ can be chosen arbitrarily small, and $r$ arbitrarily large,  
	Condition \ref{cond-suff} follows.
\end{proof}

\subsection{Double boundary action and Bowen-Margulis-Sullivan measures}\label{sec-db}

We  recall the Bader-Furman construction of Bowen-Margulis-Sullivan  measures \cite{furman-coarse,bader-furman} for $X$ Gromov-hyperbolic and $G$ acting properly discontinuously and cocompactly by isometries on it (e.g.\ $X$ may be taken to be a Cayley graph of $G$ with respect to a finite generating set). 
Let $[\mu]$ be  the Patterson-Sullivan conformal density in Example \ref{eg-confden} (1) on $\partial_G X$. The square class $[\mu\times\mu]$ is supported on  $\partial_G X \times \partial_G X$. The Gromov inner product
on $X$ with respect to the basepoint $o$ will be denoted by
$\langle{x},{y}\rangle_{o}$.
\begin{prop}\label{coarseBMS}\cite[Proposition 1]{furman-coarse}\cite[Proposition 3.3]{bader-furman}
	There exists a $G-$invariant Radon measure, denoted $\mubms$, in the measure class $[\mu\times \mu]$
	on $(\partial_G X \times \partial_G X)\setminus \Delta$, where $\Delta$ denotes the diagonal. Moreover, $\mubms$ has the form
	\[
	d\mubms(x,y)=e^{F(x,y)}\,d\mu(x)\,d\mu(y)
	\]
	where $F$ is a measurable function on $\partial_G X \times \partial_G X$
	of the form 	
	$
	F(x,y)=2v\, \langle{x},{y}\rangle_{o}+O(1),
	$ and $v$ is the Hausdorff dimension of $\mu$. 
\end{prop}

The measure $\mubms$ is called the \emph{Bowen-Margulis-Sullivan  measure.} 
\begin{prop}\label{prop-db}
	Let $G$ be a Gromov hyperbolic group acting properly discontinuously and cocompactly by isometries on $X$. Let $\mubms$  be the Bowen-Margulis-Sullivan  measure on  $(\partial_G X \times \partial_G X)\setminus \Delta$. Then 
	Condition \ref{cond-suff2} is satisfied for the diagonal $G-$action on 
	$(\partial_G X \times \partial_G X,\mubms)$. Hence,  the associated $G-$indexed $\sas$ random field  is ergodic.
\end{prop} 

\begin{proof} We need to
	check condition \ref{cond-suff2} for finite measure subsets of $\partial_G X \times \partial_G X$.
	Note that $\partial_G X \times \partial_G X\setminus \Delta$ may  be identified with
	bi-infinite geodesics in $X$ joining pairs of distinct points on $\partial_G X$.
	We further note that  finite measure subsets of  $(\partial_G X \times \partial_G X,\mubms)$ can equivalently be generated  by the collection of bi-infinite geodesics in $X$ passing through a compact subset of $X$. Let $\om_A \subset
	\partial_G X \times \partial_G X$ then denote the collection of bi-infinite geodesics in $X$ passing through a compact $A \subset X$.
	
	It thus suffices to show that for any compact $A \subset X$, Condition \ref{cond-suff2} is satisfied for $\om_A$. We shall prove the stronger assertion, as in Lemma \ref{lem-mua20} that $\mubms (\om_A \cap g. \om_A) \to 0$ as $g \to \infty$.
	The same argument as in the proof of Theorem \ref{thm-hyp} will then furnish 
	Condition \ref{cond-suff2}.
	
	To prove that $\mubms (\om_A \cap g. \om_A) \to 0$ as $g \to \infty$, observe that 
	$\mubms (\om_A \cap g. \om_A)$ is the same as the measure of geodesics passing through both $A$ and $g.A$. Let $d_g=d(A,g.A)$ denote the distance (in $X$) between $A$ and $g.A$. Let
	$\lambda(g,A)$ denote a shortest geodesic joining $A$ and $g.A$, and $l_g$ denote its length. Let $m_g$ denote the mid-point of $\lambda(g,A)$. Since $G$ acts cocompactly on $X$, there is a uniformly bounded constant $D_0$ (the diameter of a fundamental domain) and an element $h \in G$ such that $h.m_g$ lies within $D_0$ of 
	a fixed basepoint $o\in X$. Further, since $\mubms$ is $G-$invariant by Proposition
	\ref{coarseBMS}, $\mubms (\om_A \cap g. \om_A)=\mubms (h.\om_A \cap hg. \om_A)$, and 
	$\mubms (h.\om_A \cap hg. \om_A)$ is the measure of geodesics passing through both
	$h.A$ and $hg.A$.
	The distance of both $h.A$ and $hg.A$ from $o$ is at least $(\frac{1}{2}d(o,g.o)-D_0)$ as $h.\lambda_g$ is (coarsely) centered at $h.m_g$.
	Let $D_g=d(o, h.A)$. Clearly, $D_g \to \infty$ as $g \to \infty$.
	Finally, the measure $\mubms (h.\om_A \cap hg. \om_A)$ is at most $C_0 e^{-vD_g}$, where
	\begin{enumerate}
		\item $C_0$ depends only on $A$,
		\item $v$ is  the Hausdorff dimension of $\partial_G X$ with respect to the Patterson-Sullivan measure $\mu$.
	\end{enumerate}
	In particular, $\mubms (h.\om_A \cap hg. \om_A)\to 0$   as $g \to \infty$, proving our claim.
\end{proof}

\noindent {\bf Open Problems:} We conclude this paper with some open questions for stationary $\sas$ random fields indexed by amenable and hyperbolic groups. \\
1) Does Theorem \ref{t:main} hold for left-stationary max-stable random fields indexed by amenable groups?\\
2) Does Theorem \ref{t:main} hold for amenable actions (in the sense of Zimmer)?\\
3) In the  case of hyperbolic groups, does null action  characterize ergodicity of the associated $\sas$ random field? (Note that the actions in Example \ref{eg-confden} are null.)\\
4) Is it possible to characterize strong mixing for stationary $\sas$ random fields indexed by amenable groups in terms of the underlying non-singular action?\\
5) What is the role of type $III$ factors on mixing properties?

\end{document}